\documentclass[12pt]{article}
\usepackage{fullpage, amsmath, amssymb, amsfonts, amsthm, stmaryrd, url, hyperref}
\usepackage[all]{xy}

\newtheorem{theorem}{Theorem}[section]
\newtheorem{lemma}[theorem]{Lemma}
\newtheorem{cor}[theorem]{Corollary}

\newtheorem{prop}[theorem]{Proposition}
\theoremstyle{definition}
\newtheorem{defn}[theorem]{Definition}

\newtheorem{remark}[theorem]{Remark}

\newtheorem{notation}[theorem]{Notation}

\numberwithin{equation}{theorem}

\def\Q{{\mathbb {Q}}_p}
\def\F{\mathcal{F}}

\def\ra{\rightarrow}

\def\OO{\mathcal{O}}

\def\Tor{\rm{Tor}}

\def\int{\mathrm{int}}

\def\rig{\mathrm{rig}}

\def\dif{\mathrm{dif}}

\def\D{\mathrm{D}}

\def\Sen{\mathrm{Sen}}

\def\m{(\varphi,\Gamma)}

\def\r{\mathcal{R}}

\title{Semi-stable periods of finite slope families}
\author{Ruochuan Liu}
\author{Ruochuan Liu \\ Beijing International Center for Mathematical Research\\Peking Univeristy\\
Beijing, 100871\\
liuruochuan@math.pku.edu.cn}
\begin{document}
\maketitle
\begin{abstract}
We introduce the notion of finite slope families to encode the local properties of the $p$-adic families of Galois representations appearing in the work of Harris, Lan, Taylor and Thorne on the construction of Galois representations for (non-self dual) regular algebraic cuspidal automorphic representations of $\mathrm{GL}(n)$ over CM fields. Our main result is to prove the analytic continuation of semi-stable (and crystalline) periods for such families. 
\end{abstract}

\tableofcontents
\section*{Introduction and notations}
In their ICM talk \cite{SU06}, Skinner and Urban outline a program to connect the order of vanishing of the $L$-functions of certain
polarized regular motives with the rank of the associated Bloch-Kato Selmer groups.
Their strategy is to deform those motives along certain $p$-adic families, the so called eigenfamilies, to construct the expected extensions.
To this end, they introduce the notion of \emph{finite slope families} of $p$-adic representations to encode the local properties of the Galois representations arising from those $p$-adic families. One may view the finite slope families as a generalization of the $p$-adic families of Galois representations arising from the Coleman-Mazur eigencurve. In \cite{BC06}, Bellaiche and Chenevier introduce the notion of \emph{weakly refined families} of $p$-adic representations to encode the local properties of the latter. More precisely, a family of weakly refined $p$-adic representations is a family of $p$-adic representations\footnote{Strictly speaking, Bellaiche-Chenevier use pseudo representations rather than genuine representations in their definition of weakly refined families.} over a rigid analytic space with a Zariski dense subset of crystalline points which have crystalline periods of a prescribed Frobenius eigenvalue and constant Hodge-Tate weight. Moreover, this constant weight is the largest one\footnote{We normalize the Hodge-Tate weight in the way that the $p$-adic cyclotomic character has Hodge-Tate weight 1.} of all Hodge-Tate weights, and the difference of this constant weight with any other weight is unbounded over the base. For example, in the case of the eigencurve, one can take the subset of all classical eigenforms and the prescribed Fronbenius eigenvalue and constant Hodge-Tate weight are the function of $U_p$-eigenvalues and $0$ respectively. On the other hand, finite slope families generalize weakly refined families in the way that they allow
\emph{multiple} prescribed Frobenius eigenvalues and constant Hodge-Tate weights. Skinner and Urban then use the (unproved) analytic continuation of crystalline periods of finite slope families to deduce that the extensions constructed by $p$-adic deformations lie in the Selmer groups.  

Most recently, Harris, Lan, Taylor and Thorne (and Scholze independently) construct Galois representations for (non-self dual) regular algebraic cuspidal automorphic representations of $\mathrm{GL}(n)$ over CM fields \cite{HLTT}.  It turns out that these Galois representations emerge from certain $p$-adic families whose local properties generalize Skinner and Urban's finite slope families by allowing prescribed semi-stable periods. Therefore, to show that the Galois representations constructed by Harris, Lan, Taylor and Thorne have the expected properties at $p$, one needs to show the analytic continuation of semi-stable periods for those $p$-adic families.

In this paper, we make use of the notion of finite slope families to encode the local properties of the $p$-adic families of Galois representations appearing in the work of Harris, Lan, Taylor and Thorne; this generalizes the original definition of Skinner and Urban. Our main result is then to prove the analytic continuation of semi-stable periods for such families. This will provide a necessary ingredient to Skinner and Urban's ICM program. Besides, we recently learned from Taylor that, in an ongoing project of Ila Varma, she will establish the expected properties of those Galois representations based on the results of this paper. We also note that recently Shah proves some results about interpolating Hodge-Tate and de Rham periods in families of $p$-adic Galois representations which may be applied to some related situations \cite{S}.

In the following, we state our main results precisely.
We fix a finite extension $K$ of $\Q$.
Let $K_0$ be the maximal unramified sub-extension of $K$, and let $f=[K_0:\Q]$. We also fix a finite extension $F$ of $\Q$ contained in $\overline{\mathbb{Q}}_p$ such that $\mathrm{Hom}(K, F)=\mathrm{Hom}(K,\overline{\mathbb{Q}}_p)$; here $\mathrm{Hom}$ denotes the set of $\Q$-algebra homomorphisms.

\begin{defn}\label{def:fs}
Let $X$ be a reduced rigid analytic space over $F$. A \emph{finite slope family} of $p$-adic representations of dimension $d$ over $X$ is a locally free coherent $\OO_X$-module $V_X$ of rank $d$ equipped with a continuous $G_K$-action and together with the following data
\begin{enumerate}
\item[(1)]positive integers $b, c$,
\item[(2)]a monic polynomial $Q(T)\in\OO_X(X)[T]$ of degree $m$ with unit constant term,
\item[(3)]a subset $Z$ of $X$ such that for all $z$ in $Z$, $V_z$ is semi-stable with non-positive Hodge-Tate weights, and for all $B\in\mathbb{Z}$ the set of $z$ in $Z$ such
that $V_z$ has $d-c$ Hodge-Tate weights less than $B$ is Zariski dense in $X$,
 \item[(4)]for $z\in Z$, a $K_0\otimes_{\Q}k(z)$-direct summand $\mathcal{F}_{z}$ of $D^+_{\mathrm{st}}(V_z)$ which is free of rank $c$ and stable under $\varphi$ and $N$ such that $\varphi^f$ has characteristic polynomial $Q(z)(T)$ and all Hodge-Tate weights of $\mathcal{F}_z$ lie in $[-b,0]$.
\end{enumerate}

We also need to extend the functors $D_{\mathrm{crys}}^+$ and $D_{\mathrm{st}}^+$ to families of $p$-adic representations over rigid analytic spaces.

\begin{defn}\label{def:crys-st-family}
Let $X$ be a rigid analytic space over $\Q$, and let $V_X$ be locally free coherent $\OO_X$-module equipped with a continuous $G_K$-action. Define $D_{\mathrm{crys}}^+(V_X)$ and $D_{\mathrm{st}}^+(V_X)$ to be the presheaves
\[
M(S)\mapsto D_{\mathrm{crys}}^+(V_S)=(V_S\widehat{\otimes}_{\Q}\mathbf{B}_{\mathrm{crys}}^+)^{G_K}
\]
and
\[
M(S)\mapsto D_{\mathrm{st}}^+(V_S)=(V_S\widehat{\otimes}_{\Q}\mathbf{B}_{\mathrm{st}}^+)^{G_K},
\]
where $M(S)$ runs through all admissible affinoid subdomain of $X$, respectively; here $V_S$ is the restriction of $V_X$ on $M(S)$.
\end{defn}

Now we can state our main result precisely. 
\begin{theorem}\label{thm:main}
Let $V_X$ be a finite slope family over $X$. Then there exists a surjective proper morphism $X'\ra X$ so that $(K\otimes_{K_0}D^+_{\mathrm{st}}(V_{X'}))^{Q(\varphi^f)=0}$ has a rank $c$ locally free coherent $K_0\otimes_{\Q}\OO_{X'}$-submodule which specializes to a rank $c$ free $K_0\otimes_{\Q}k(x)$-submodule in $D_\mathrm{st}^+(V_x)$ for any $x\in X'$. As a consequence, $D^+_{\mathrm{st}}(V_x)^{Q(x)(\varphi^f)=0}$ has a free $K_0\otimes_{\Q}k(x)$-submodule of rank $c$ for any $x\in X$.
\end{theorem}

The following result follows immediately. 

\begin{cor}
Let $V_X$ be a finite slope family over $X$. If $V_z$ is crystalline for any $z\in Z$, then there exists a surjective proper morphism $X'\ra X$ so that $(K\otimes_{K_0}D^+_{\mathrm{crys}}(V_{X'}))^{Q(\varphi^f)=0}$ has a rank $c$ locally free coherent $K_0\otimes_{\Q}\OO_{X'}$-submodule which specializes to a rank $c$ free $K_0\otimes_{\Q}k(x)$-submodule in $D_{\mathrm{crys}}^+(V_x)$ for every $x\in X'$. As a consequence, $D^+_{\mathrm{crys}}(V_x)^{Q(x)(\varphi^f)=0}$ has a free $K_0\otimes_{\Q}k(x)$-submodule of rank $c$ for every $x\in X$.
\end{cor}

Since the families of $p$-adic representations arising from the Coleman-Mazur eigencurve are special cases of finite slope families, Theorem \ref{thm:main} generalizes the famous result of Kisin on the analytic continuation of crystalline periods over the eigencurve \cite{Ki03}. However, even in the case that the prescribed periods are crystalline, our method is completely different from his. In fact, in the work of Kisin as well as the recent enhancement made by us \cite{L12}, one crucially uses the fact that the families of $p$-adic representations arising from the eigencurve have only one constant Hodge-Tate weight. On the other hand, our strategy and techniques are largely inspired by the works of Berger and Colmez on families of de Rham representations with bounded Hodge-Tate weights \cite{BC07} and Kedlaya, Pottharst and Xiao on the cohomology of families of $\m$-modules \cite{KPX}.  In fact, for a finite slope family, we first construct a sub-family of $\m$-modules interpolating the prescribed semi-stable periods, after making a proper and surjective base change.  This is achieved by adapting some techniques of \cite{KPX}.  This sub-family of $\m$-modules is expected to be semi-stable and produce the desired semi-stable periods. However, we are unable to prove this directly due to some technical obstacle. Instead,  we first show that this sub-family of $\m$-modules is de Rham using the fact that it is de Rham at a Zariski dense subset of the base. To this end,  we develop a theory of families of Hodge-Tate and de Rham $\m$-modules with bounded Hodge-Tate weights, which generalizes the theory of families of Hodge-Tate and de Rham representations with bounded Hodge-Tate weights developed in \cite{BC07}.  Then we prove the $p$-adic local monodromy for the restrictions of families of de Rham $\m$-modules with bounded Hodge-Tate weights on their Shilov boundaries by mimicing the proof for families of de Rham representations with bounded Hodge-Tate weights given in [\emph{loc. cit.}].  This implies that the de Rham periods of this sub-family of $\m$-modules become potentially semi-stable after restricting on the Shilov boundary. Finally, we use a key lemma due to Berger and Colmez [\emph{loc. cit.}] to conclude that these de Rham periods are actually semi-stable. 


\end{defn}

\section*{Notations}
We choose a compatible sequence of primitive $p$-powers roots of unity $(\varepsilon_n)_{n\geq0}$, i.e. each $\varepsilon_n\in\overline{\mathbb{Q}}_p$ is a primitive $p^n$-th root of $1$, and they satisfy $\varepsilon^p_{n+1}=\varepsilon_{n}$ for all $n\geq0$. Fix $\varepsilon=(\varepsilon_0,\varepsilon_1,\dots)$, and let $t=\log [\varepsilon]$ be Fontaine's $p$-adic $2\pi i$.  For a finite extension $L$ of $\Q$ in $\mathbb{C}_p$, let $L_n=L(\varepsilon_n)$ for $n\geq 1$, and let $L_\infty=\cup_{n\in\mathbb{N}}L_n$. Let $L_0'$ be the maximal unramified extension of $\Q$ in $L_\infty$. Let $\Gamma_L=\mathrm{Gal}(L_{\infty}/L)$ and $\Gamma_{L_n}=\mathrm{Gal}(L_{\infty}/L_n)$ for $n\geq 1$. For simplicity, denote $\Gamma_K$ and $\Gamma_{K_n}$ by $\Gamma$ and $\Gamma_n$ respectively. Let $\chi$ denote the $p$-adic cyclotomic character. For a $p$-adic representation $V$ of $G_K$ and $n\in\mathbb{Z}$, we set $V(n)=V\otimes \chi^n$. For $n\geq0$, let $r_n=p^{n-1}(p-1)$. For $s>0$, let $n(s)$ be the maximal integer $n$ such that $r_n\leq s$. 

\section*{Acknowledgements}
Thanks to Christopher Skinner, Richard Taylor and Ila Varma for useful communications. We especially thank Richard Taylor for suggesting a more concise definition of finite slope families.

\section{Families of $\m$-modules}
In the section we recall the notion of families of $\m$-modules over rigid analytic spaces.  For the  period rings involved in this paper, we follow the notations introduced in \cite{LB02}, and we refer the reader to [\emph{loc. cit.}] for their precise definitions. Note that this is different from the ``Robba ring" type notations used in \cite{KPX}. 
A good dictionary for these two types of notations is given in \cite[\S1]{LB08}. 
\begin{defn}
Let $A$ be a Banach algebra over $\Q$. For $s>0$, a \emph{$\varphi$-module} over $\mathbf{B}^{\dag,s}_{\rig,K}\widehat{\otimes}_{\Q}A$ is a finite projective $\mathbf{B}^{\dag,s}_{\rig,K}\widehat{\otimes}_{\Q}A$-module $D_A^s$ equipped with an isomorphism
$$\varphi^*D_A^s\cong D_A^s\otimes_{\mathbf{B}^{\dag,s}_{\rig,K}\widehat{\otimes}_{\Q}A}\mathbf{B}^{\dag,ps}_{\rig,K}\widehat{\otimes}_{\Q}A.$$ A \emph{$\varphi$-module} $D_A$ over $\mathbf{B}^{\dag}_{\rig,K}\widehat{\otimes}_{\Q}A$ is the base change to $\mathbf{B}^{\dag}_{\rig,K}\widehat{\otimes}_{\Q}A$ of a $\varphi$-module $D_A^s$ over $\mathbf{B}^{\dag,s}_{\rig,K}\widehat{\otimes}_{\Q}A$ for some $s>0$. 
A \emph{$\m$-module} over $\mathbf{B}^{\dag,s}_{\rig,K}\widehat{\otimes}_{\Q}A$ is a $\varphi$-module $D_A^s$ over $\mathbf{B}^{\dag,s}_{\rig,K}\widehat{\otimes}_{\Q}A$ equipped with a commuting $\mathbf{B}^{\dag,s}_{\rig,K}$-semilinear and $A$-linear continuous action of $\Gamma$.  A \emph{$\m$-module} $D_A$ over $\mathbf{B}^{\dag}_{\rig,K}\widehat{\otimes}_{\Q}A$ is the base change to $\mathbf{B}^{\dag}_{\rig,K}\widehat{\otimes}_{\Q}A$ of a $\m$-module $D_A^s$ over $\mathbf{B}^{\dag,s}_{\rig,K}\widehat{\otimes}_{\Q}A$ for some $s>0$.
\end{defn}

\begin{notation}
For a morphism $A\ra B$ of Banach algebras over $\Q$, we denote by $D^s_B$ (resp. $D_B$) the base change of $D^s_A$ (resp. $D_A$) to $\mathbf{B}^{\dag,s}_{\rig,K}\widehat{\otimes}_{\Q}B$ (resp. $\mathbf{B}^{\dag}_{\rig,K}\widehat{\otimes}_{\Q}B$). In the case when $A=S$ is an affinoid algebra over $\Q$ and $x\in M(S)$, we denote $D^s_{k(x)}$ (resp. $D_{k(x)}$) by $D_x^s$ (resp. $D_x$) instead.
\end{notation}

To define $\m$-modules over general rigid analytic spaces, one needs to show that $\varphi$-modules over affinoid spaces satisfy the gluing property.  To this end, we recall the notion of $\varphi$-bundles introduced in \cite{KPX}. Let $S$ be an affinoid algebra over $\Q$. For $0<s_1<s_2$ which are sufficiently large, a vector bundle over $\mathbf{B}^{[s_1, s_2]}_{K}\widehat{\otimes}_{\Q}S$ a finite projective module $D_S^{[s_1,s_2]}$ over $\mathbf{B}^{[s_1,s_2]}_K\widehat{\otimes}_{\Q}S$. By the identification of $\mathbf{B}_K^{[s_1,s_2]}$ with the ring of rigid analytic functions over the closed annulus $s_1\leq v_p(T)\leq s_2$ over $K_0'$, one may identify $D_S^{[s_1,s_2]}$ with a locally free coherent sheaf over the product of the annulus $s_1\leq v_p(T)\leq s_2$ over $K'_0$ with $M(S)$ in the category of rigid analytic spaces over $\Q$. It then follows that vector bundles over $\mathbf{B}^{[s_1, s_2]}_{K}\widehat{\otimes}_{\Q}S$ satisfy the glueing property for the weak $G$-topology of $M(S)$. For sufficiently large $s$, a vector bundle over $\mathbf{B}^{\dag,s}_{\rig,K}\widehat{\otimes}_{\Q}S$ consists of one vector bundle $D_S^{[s_1,s_2]}$ over each ring $\mathbf{B}^{[s_1,s_2]}_K\widehat{\otimes}_{\Q}S$ with $s\leq s_1\leq s_2$, together with isomorphisms
\[
D_S^{[s_1,s_2]}\otimes_{\mathbf{B}^{[s_1,s_2]}_{K}\widehat{\otimes}_{\Q}S}
\mathbf{B}^{[s_1',s_2']}_{K}\widehat{\otimes}_{\Q}S\cong D_S^{[s'_1,s'_2]}
\]
for all $s\leq s_1'\leq s_1\leq s_2\leq s_2'$ satisfying the cocycle conditions. A $\varphi$-bundle over $\mathbf{B}^{\dag,s}_{\rig,K}\widehat{\otimes}_{\Q}S$ is a vector bundle $(D_S^{[s_1,s_2]})_{s\leq s_1\leq s_2}$ over $\mathbf{B}^{\dag,s}_{\rig,K}\widehat{\otimes}_{\Q}S$ equipped with isomorphisms $\varphi^*D_S^{[s_1,s_2]}\cong D_S^{[ps_1,ps_2]}$ for all $s\leq s_1\leq s_2$ satisfying the obvious compatibility conditions. When $s$ is sufficiently large, by \cite[Proposition 2.2.7]{KPX}, the natural functor from the category of $\varphi$-modules over $\mathbf{B}^{\dag,s}_{\rig,K}\widehat{\otimes}_{\Q}S$ to the category of $\varphi$-bundles over $\mathbf{B}^{\dag,s}_{\rig,K}\widehat{\otimes}_{\Q}S$ is an equivalence of categories. Note that by the glueing property of vector bundles, one can glue $\varphi$-bundles $\mathbf{B}^{\dag,s}_{\rig,K}\widehat{\otimes}_{\Q}S$ over $M(S)$. Therefore this equivalence of categories enables us to glue $\varphi$-modules over affinoid spaces.

\begin{defn}
Let $X$ be a rigid analytic space over $\Q$. A family of $\m$-modules $D_X$ over $X$ is a compatible family of $\m$-modules $D_S$ over $\mathbf{B}^{\dag}_{\rig,K}\widehat{\otimes}_{\Q}S$ for each affinoid subdomain $M(S)$ of $X$. By the gluing property of $\varphi$-modules over affinoid spaces, one may view $D_X$ as a sheaf over $X$ for the weak $G$-topology (hence extends uniquely to the strong $G$-topology).
\end{defn}

\begin{theorem}
Let $A$ be a Banach algebra over $\Q$, and let $V_A$ be a finite locally free $A$-linear representation of $G_K$. Then there is a $\m$-module $\D_\rig^\dag(V_A)$ over $\mathbf{B}^{\dag}_{\rig,K}\widehat{\otimes}_{\Q}A$ functorially associated to $V_A$. The rule $V_A\mapsto \D_\rig^\dag(V_A)$ is fully faithful and exact, and it commutes with base change in $A$.
\end{theorem}
\begin{proof}
See \cite[Theorem 3.11]{KL10}, which generalizes \cite[Th\'eor\`eme 4.2.9]{BC07}. Note that both the results do not really verify the $\varphi$-module condition. This gap is fixed by \cite[Theorem 1.1.4]{L12}.
\end{proof}

Let $A$ be a Banach algebra over $K_0$. Recall that one has a canonical decomposition
\[
A\otimes_{\Q}K_0\cong\prod_{\sigma\in\mathrm{Gal}(K_0/\Q)}A_{\sigma}
\]
where each $A_{\sigma}$ is the base change of $A$ by the automorphism $\sigma$. Furthermore, the $\mathrm{Gal}(K_0/\Q)$-action permutes all $A_\sigma$'s in the way that $\tau(A_\sigma)=A_{\tau\sigma}$. For any $a\in A^\times$, we equip $A\otimes_{\Q}{K_0}$ with a $1\otimes \varphi$-semilinear action $\varphi$ by setting
\[
\varphi((x_1,x_{\varphi},\dots, x_{\varphi^{f-1}}))=(ax_{\varphi^{f-1}},x_1,\dots,x_{\varphi^{f-2}})
\]
where $\varphi$ is the geometric Frobenius and $x_{\varphi^i}\in A_{\varphi^i}$ for each $0\leq i\leq f-1$; we denote this $\varphi$-module by $D_a$. It is clear that the $\varphi$-action on $D_a$ satisfies $\varphi^f=1\otimes a$.

We fix a uniformizer $\pi_K$ of $K$.
\begin{defn} For any continuous character $\delta:K^\times\ra A^\times$, we define a rank 1 $(\varphi,\Gamma)$-module $(\mathbf{B}_{\rig,K}^\dag\widehat{\otimes}_{\Q}A)(\delta)$ over $\mathbf{B}_{\rig,K}^\dag\widehat{\otimes}_{\Q}A$ as follows. If $\delta|_{\OO_K^\times}=1$, we set 
\[
(\mathbf{B}_{\rig,K}^\dag\widehat{\otimes}_{\Q}A)(\delta)=(\mathbf{B}_{\rig,K}^\dag\widehat{\otimes}_{\Q}A)
\otimes_{A\otimes_{\Q}{K_0}}D_{\delta(\pi_K)}
\] 
where we equip $D_{\delta(\pi_K)}$ with the trivial $\Gamma$-action. For general $\delta$, we write $\delta=\delta'\delta''$ such that $\delta'(\pi_K)=1$ and $\delta''|_{\OO_K^\times}=\mathrm{id}$. We view $\delta'$ as an $A$-valued character of $W_K$ via the local reciprocity map, and extend it to a character of $G_K$ continuously. We then set
\[
(\mathbf{B}_{\rig,K}^\dag\widehat{\otimes}_{\Q}A)(\delta)=\D_\rig^\dagger(\delta')
\otimes_{\mathbf{B}_{\rig,K}^\dag\widehat{\otimes}_{\Q}A}
(\mathbf{B}_{\rig,K}^\dag\widehat{\otimes}_{\Q}A)(\delta'').
\]
For any $(\varphi,\Gamma)$-module $D_A$ over $\mathbf{B}_{\rig,K}^\dag\widehat{\otimes}_{\Q}A$, put $D_A(\delta)=D_A\otimes_{\mathbf{B}_{\rig,K}^\dag\widehat{\otimes}_{\Q}A}
(\mathbf{B}_{\rig,K}^\dag\widehat{\otimes}_{\Q}A)(\delta)$.

Let $X$ be a rigid analytic space over $\Q$. For a continuous character $\delta:K^\times\ra \OO(X)^\times$ and a family of $\m$-module $D_X$ over $X$, we define the families of $\m$-modules $(\mathbf{B}_{\rig,K}^\dag\widehat{\otimes}_{\Q}\OO_X)(\delta)$ and $D_X(\delta)$ by gluing  $(\mathbf{B}_{\rig,K}^\dag\widehat{\otimes}_{\Q}S)(\delta)$ and $D_S(\delta)$ for all affinoid subdomains $M(S)$ respectively.
\end{defn}

\section{Cohomology of families of $\m$-modules}
Let $\Delta_K$ be the $p$-torsion subgroup of $\Gamma$. Choose $\gamma_K\in\Gamma_K$ whose image in $\Gamma/\Delta_K$ is a topological generator.
\begin{defn}
For a $\m$-module $D_S$ over $\mathbf{B}_{\rig,K}^\dag\widehat{\otimes}_{\Q}S$, we define the Herr complex $C^\bullet_{\varphi,\gamma_K}(D_S)$ of $D_S$ concentrated in degree $[0,2]$ as follows:
\[
  C^{\bullet}_{\varphi,\gamma_K}(D_S)=
 [D_S^{\Delta_K}\stackrel{d_{1}}{\longrightarrow}D_S^{\Delta_K}\oplus D_S^{\Delta_K}
 \stackrel{d_{2}}{\longrightarrow}D_S^{\Delta_K}]
\]
with $d_1(x) = ((\gamma_K - 1)x, (\varphi - 1)x)$ and $d_2(x,y) =
(\varphi - 1)x - (\gamma_K - 1)y$. One shows that this complex is independent of the choice of $\gamma_K$ up to canonical $S$-linear isomorphisms: the isomorphism $C^{\bullet}_{\varphi,\gamma_K}(D_S)\stackrel{\sim}{\to}C^{\bullet}_{\varphi,\gamma'_K}(D_S)$ is given by $[1, 1\oplus \frac{\gamma'_K-1}{\gamma_K-1}, \frac{\gamma'_K-1}{\gamma_K-1}]$. We denote the cohomology of $C^{\bullet}_{\varphi,\gamma_K}(D_S)$ by $H^\bullet(D_S)$.
\end{defn}

By the main result of \cite{KPX}, one knows that $H^i(D_S)$ is a finitely generated $S$-module.  It therefore follows that $H^i(D_S)$ commutes with flat base change in $S$. That is, if $S\ra S'$ is flat, then $H^i(D_S)\otimes_SS'\cong H^i(D_{S'})$. This enables a cohomology theory for families of $\m$-modules over general rigid analytic spaces.

\begin{defn}
Let $X$ be a rigid analytic space over $\Q$, and let $D_X$ be a family of $\m$-modules over $X$. We define $H^\bullet(D_X)$ to be the cohomology of the complex of sheaves
\[
C^{\bullet}_{\varphi,\gamma_K}(D_X)=
 [D_X^{\Delta_K}\stackrel{d_{1}}{\longrightarrow}D_X^{\Delta_K}\oplus D_X^{\Delta_K}
 \stackrel{d_{2}}{\longrightarrow}D_X^{\Delta_K}]
\]
in the category of presheaves over $X$ with $d_1(x) = ((\gamma_K - 1)x, (\varphi - 1)x)$ and $d_2(x,y) =
(\varphi - 1)x - (\gamma_K - 1)y$. For each affinoid subdomain $M(S)$ of $X$ and $0\leq i\leq 2$, the module of sections of $H^i(D_X)$ on $M(S)$ is canonically isomorphic to $H^i(D_S)$. Hence $H^i(D_X)$ forms a coherent $\OO_X$-module by the flat base change property of $H^i(D_S)$.
\end{defn}

As a consequence of finiteness of the cohomology of families of $\m$-modules, by a standard argument we see that locally on $X$, the complex $C^{\bullet}_{\varphi,\gamma_K}(D_X)$ is quasi isomorphic to a complex of locally free coherent sheaves concentrated in degree $[0,2]$. This enables us to flatten the cohomology of families of $\m$-modules by blowing up the base $X$. The following lemma is a rearrangement of some arguments in \cite[\S6.3]{KPX}.

\begin{lemma}\label{lem:modification}
Let $X$ be a reduced and irreducible rigid analytic space over $F$, and let $D_X$ be a family of $\m$-modules of rank $d$ over $X$. Then the following are true.
\begin{enumerate}
\item[(1)]There exists a proper birational morphism $\pi:X'\ra X$ of reduced rigid analytic spaces over $F$ so that  $H^0(D_{X'})$ is flat and $H^i(D_{X'})$ has Tor-dimensions $\leq 1$ for each $i=1,2$.
\item[(2)]Suppose that $D'_{X}$ is a family of $\m$-modules over $X$ of rank $d'$, and that $\lambda: D'_X\ra D_X$ be a morphism between them so that for any $x\in X$, the image of $\lambda_x$ is a $\m$-submodule of rank $d$ of $D_x$. Then there exists a proper birational morphism $\pi:X'\ra X$ of reduced rigid analytic spaces over $F$ so that  the cokernel of $\pi^*\lambda$ has Tor-dimension $\leq 1$.
\end{enumerate}
\end{lemma}
\begin{proof}
The upshot is that for a bounded complex $(C^\bullet,d^\bullet)$ of locally free coherent sheaves on $X$, there exists a blow up $\pi:X'\ra X$, which depends only on the quasi-isomorphism class of $(C^\bullet,d^\bullet)$, so that $\pi^*d^i$ has flat image for each $i$. Furthermore, the construction of $X'$ commutes with dominant base change in $X$ (see \cite[Corollary 6.3.6]{KPX} for more details). Thus for (1), we can construct $X'$ locally and then glue. For (2),
let $Q_X$ denote the cokernel of $\lambda$. For any $x\in X$, since the image of $\lambda_x$ is a $\m$-submodule of rank $d$, by \cite[Lemma 5.3.1]{L12}, we get that $Q_x$ is killed by a power of $t$. Now let $M(S)$ be an affinoid subdomain of $X$, and suppose that $D_S^s$ and $D'^s_S$ are defined for some suitable $s>0$. For $r>s$, set $Q_S^{[s,r]}=D^{[s,r]}_S/\lambda(D'^{[s,r]}_S)$. Since for any $x\in M(S)$, the fiber of $Q_S^{[s,r]}$ at $x$ is killed by a power of $t$, we get that $Q_S^{[s,r]}$ is killed by $t^k$ for some $k>0$. This yields that $Q_S^{[s,r]}$ is a finite $S$-module. Now we apply \cite[Corollary 6.3.6]{KPX} to a finite presentation of $Q_S^{[s,ps]}$ to get a blow up $Y$ of $M(S)$ so that the pullback of $Q_S^{[s,ps]}$ has Tor-dimension $\leq1$. Using the fact $(\varphi^n)^*Q_S^{[s,ps]}\cong Q_S^{[p^ns,p^{n+1}s]}$, we see that $Y$ is also the blow up obtained by applying \cite[Corollary 6.3.6]{KPX} to a finite presentation of $Q_S^{[s,p^{n+1}s]}$ for any positive integer $n$. It therefore follows that for any $r>s$, the pullback of $Q_S^{[s,r]}$ has Tor-dimension $\leq 1$; hence the pullback of $Q_S$ has Tor-dimension $\leq 1$. Furthermore, the blow ups for all affinoid subdomains $M(S)$ glue to form a blow up $X'$ of $X$ which satisfies the desired condition.
\end{proof}

\begin{lemma}\label{lem:ker-birational}
Let $X$ be a reduced and irreducible rigid analytic space over $F$. Let $D'_X$ and $D_{X}$ be families of $\m$-modules over $X$ of ranks $d'$ and $d$ respectively, and let $\lambda: D'_X\ra D_X$ be a morphism between them. Suppose that for any $x\in X$, the image of $\lambda_x$ is a $\m$-submodule of rank $d$ of $D_x$. Then there exists a proper birational morphism $\pi:X'\ra X$ of reduced rigid analytic spaces over $F$ such that the kernel of $\pi^*\lambda$ is a family of $\m$-modules of rank $d'-d$ over $X'$, and there exists a Zariski open dense subset $U\subset X'$ such that $(\ker(\pi^*\lambda))_x=\ker((\pi^*\lambda)_x)$ for any $x\in U$.
\end{lemma}
\begin{proof}
Let $Q_X$ be the cokernel of $\lambda$. By the previous Lemma, we may suppose that $Q_X$ has Tor-dimension $\leq1$ after adapting $X$. Now let $P_X$ denote the kernel of $\lambda$. For any $x\in X$, the Tor spectral sequence computing the cohomology of the complex $[D_{X}\stackrel{\lambda}{\longrightarrow}D'_{X}]\otimes^{\mathbf{L}}_{\OO_{X}}k(x)$ gives rise to a short exact sequence
\[
0\longrightarrow P_x\longrightarrow\ker(\lambda_x)\longrightarrow\mathrm{Tor}_1(Q_X,k(x))\longrightarrow0.
\]
Since the image of $\lambda_x$ is a $\m$-module of rank $d$, $\ker(\lambda_x)$ is a $\m$-module of rank $d'-d$. Since $Q_X$ is killed by a power of $t$ locally on $X$, we get that the last term of the exact sequence is killed by a power of $t$. This yields that $P_x$ is a $\m$-module of rank $d'-d$. We therefore conclude that $P_X$ is a family of $\m$-modules of rank $d'-d$ over $X$ by \cite[Corollary 2.1.9]{KPX}. Furthermore, since $Q_X$ has Tor-dimension $\leq1$, by \cite[Lemma 6.3.7]{KPX}, we get that the set of $x\in X$ for which $\mathrm{Tor}_1(Q_X,k(x))\neq0$ forms a nonwhere dense Zariski closed subset of $X$; this yields the rest of the lemma.
\end{proof}

The following proposition modifies part of \cite[Theorem 6.3.9]{KPX}.

\begin{prop}\label{prop:cohomology}
Let $X$ be a reduced and irreducible rigid analytic space over $F$. Let $D_X$ be a family of $\m$-modules of rank $d$ over $X$, and let $\delta:K^\times\ra \OO(X)^\times$ be a continuous character. Suppose that there exists a Zariski dense subset $Z$ of closed points of $X$ and a positive integer $c\leq d$ such that for every $z\in Z$, $H^0(D_z^{\vee}(\delta_z))$ is a
$c$-dimensional $k(z)$-vector space.
Then there exists a proper birational morphism $\pi:X'\ra X$ of reduced rigid analytic spaces over $F$ and a morphism $\lambda: D_{X'}\ra M_{X'}=(\mathbf{B}_{\rig,K}^\dag\widehat{\otimes}_{\Q}\OO_{X'})(\delta)\otimes_{\OO_{X'}}L$ of $\m$-modules, where $L$ is a locally free coherent $\OO_{X'}$-module of rank $c$ equipped with trivial $\varphi,\Gamma$-actions, such that
\begin{enumerate}
\item[(1)]for any $x\in X'$, the image of $\lambda_{x}$ is a $\m$-submodule of rank $c$;
\item[(2)]the kernel of $\lambda$ is a family of $\m$-modules of rank $d-c$ over $X'$, and there exists a Zariski open dense subset $U\subset X'$ such that $(\ker\lambda)_x=\ker(\lambda_x)$ for any $x\in U$.
\end{enumerate}
\end{prop}
\begin{proof}
Using Lemma \ref{lem:modification}, we first choose a proper birational morphism $\pi:X'\ra X$ with $X'$ reduced such that $N_{X'}=\pi^*(D^{\vee}_{X}(\delta))$ satisfies the conditions that $H^0(N_{X'})$ is flat and $H^i(N_{X'})$ has Tor-dimension $\leq 1$ for each $i=1,2$. Then for any $x\in X'$, the base change spectral sequence $E^{i,j}_2=\mathrm{Tor}_{-i}(H^j(N_{X'}),k(x))\Rightarrow H^{i+j}(N_x)$ gives a short exact sequence (using that $H^1(N_{X'})$ has $\Tor$-dimension $\leq 1$ and $N_{X'}$ is flat)
\[
0\longrightarrow H^0(N_{X'})\otimes_{\OO_{X'}}k(x)\longrightarrow H^0(N_x)\longrightarrow \mathrm{Tor}_1(H^1(N_{X'}),k(x))\longrightarrow0.
\]
Since $H^1(N_{X'})$ has Tor-dimension $\leq1$, by \cite[Lemma 6.3.7]{KPX}, the set of $x\in X'$ for which the last term of the above exact sequence does not vanish forms a nowhere dense Zariski closed subset $V$. For any $z\in\pi^{-1}(Z)\setminus V$, we deduce that $H^0(N_{X'})\otimes_{\OO_{X'}}k(z)$ is a $c$-dimensional $k(z)$-vector space. Since $H^0(N_{X'})$ is flat and $\pi^{-1}(Z)\setminus V$ is a Zariski dense subset of $X'$, we get that $H^0(N_{X'})$ is locally free of constant rank $c$. Let $L$ be the dual coherent $\OO_{X'}$-module of it. Then the natural map $(\mathbf{B}_{\rig,K}^\dag\widehat{\otimes}_{\Q}\OO_{X'})H^0(N_{X'})\ra N_{X'}$
gives rise to a map 
\[
\lambda:D_{X'}\ra M_{X'}=(\mathbf{B}_{\rig,K}^\dag\widehat{\otimes}_{\Q}\OO_{X'})(\delta)\otimes_{\OO_{X'}}L.
\] 
For any $x\in X'$, since the map $H^0(N_{X'})\otimes_{\OO_{X'}}k(x)\rightarrow H^0(N_x)$ is injective, we get that the image of $\lambda_x$ is a rank $c$ $\m$-submodule of $M_x$. We thus conclude the proposition using the previous lemma.
\end{proof}

\section{Families of Hodge-Tate $\m$-modules}
From now on, let $S$ be an affinoid algebra over $F$. Recall that for any $n\geq n(s)$, there is a continuous $\Gamma$-equivariant injective map 
\[
\iota_n:\mathbf{B}_{\rig, K}^{\dagger,s}\ra K_n[[t]].
\]
It is defined as the composite
\[
\xymatrix{
\mathbf{B}_{K}^{\dagger,s}\subset\widetilde{\mathbf{B}}^{\dag,s}\stackrel{\varphi^{-n}}{\to}\widetilde{\mathbf{B}}^{\dag,p^{-n}s}\subset \widetilde{\mathbf{B}}^+\subset\mathbf{B}_{\mathrm{dR}}^+,}
\]
and it factors through $K_n[[t]]$  (see \cite[\S2]{LB02} for more details about $\iota_n$).  In particular we have $\iota_{n+1}\circ\varphi=\iota_n$. The map $\iota_n$ induces a continuous $\Gamma$-equivariant map 
\[
\iota_n: \mathbf{B}_{\rig,K}^{\dagger,s}\widehat{\otimes}_{\Q}S\ra K_n[[t]]\widehat{\otimes}_{\Q}S.
\] 

\begin{defn}
Let $D_S$ be a $\m$-module of rank $d$ over $\mathbf{B}_{\rig,K}^\dag\widehat{\otimes}_{\Q}S$.  For any positive integer $n$, if $D_S^{r_n}$ is defined, then for any $0<s\leq r_n$, we set
\[
\D_{\dif}^{+,K_n}(D_S)=D^{s}_S\otimes_{\mathbf{B}_{\rig,K}^{\dagger,s}\widehat{\otimes}_{\Q}S,\iota_n}(K_n[[t]]\widehat{\otimes}_{\Q}S)
\]
and
\[
\D_{\dif}^{K_n}(D_S)=\D_{\dif}^{+,K_n}(D_S)[1/t].
\]
We also denote the natural map 
\[
D^{s}_S\ra\D_{\dif}^{+,K_n}(D_S),
\]
by $\iota_n$, and call it the \emph{localization map}. Define $\D^{K_n}_{\Sen}(D_S)=\D^{+,K_n}_{\dif}(D_S)/(t)$. For $D_S=\D_{\rig}^\dag(V_S)$ coming from a finite locally free $S$-linear representation $V_S$, we write $\D^{+,K_n}_{\dif}(V_S)$ and $\D^{K_n}_{\Sen}(V_S)$ for $\D^{+,K_n}_{\dif}(D_S)$ and $\D^{K_n}_{\Sen}(D_S)$ respectively. When the base field is clear from the context, we write  $\D^{+,n}_{\dif}(D_S)$ and $\D^{n}_{\Sen}(D_S)$ instead of $\D^{+,K_n}_{\dif}(D_S)$ and $\D^{K_n}_{\Sen}(D_S)$ for simplicity. 
\end{defn}

\begin{defn}\label{def:HT}
We call $D_S$ \emph{Hodge-Tate with Hodge-Tate weights in $[a,b]$} if
there exists a positive integer $n$ such that
the natural map
\begin{equation}\label{eq:def-HT}
(\oplus_{a\leq i\leq b}\D^n_\Sen(D_S(-i)))^\Gamma\otimes_{K\otimes_{\Q}S}(K_n\otimes_{\Q}S)[t,t^{-1}]\longrightarrow \oplus_{i\in\mathbb{Z}}\D_\Sen^n(D_S(-i))
\end{equation}
is an isomorphism. We denote by $h_{HT}(D_S)$ the smallest $n$ which satisfies this condition, and we define $D_{\mathrm{HT}}(D_S)=(\oplus_{a\leq i\leq b}\D^{h_{HT}(D_S)}_\Sen(D_S(-i)))^\Gamma$.
\end{defn}

\begin{lemma}\label{lem:HT-inv}
Let $D_S$ be a Hodge-Tate $\m$-module over $\mathbf{B}_{\rig,K}^\dag\widehat{\otimes}_{\Q}S$ with weights in $[a,b]$. Then for any $n\geq h_{HT}(D_S)$, (\ref{eq:def-HT}) is an isomorphism and $\D_\Sen^n(D_S(-i))^{\Gamma}=\D_\Sen^{h_{HT}(D_S)}(D_S(-i))^{\Gamma}$ for any $i\in [a,b]$. As a consequence, we have
$(\oplus_{a\leq i\leq b}\D_\Sen^n(D_S(-i)))^\Gamma=D_{\mathrm{HT}}(D_S)$.
\end{lemma}
\begin{proof}
Tensoring with $K_{n}\otimes_{\Q}S[t,1/t]$ on both sides of the map
\[
(\oplus_{a\leq i\leq b}\D^{h_{HT}(D_S)}_\Sen(D_S(-i)))^\Gamma\otimes_{K\otimes_{\Q}S}(K_{h_{HT}(D_S)}
\otimes_{\Q}S)[t,t^{-1}]\longrightarrow \oplus_{i\in\mathbb{Z}}\D_\Sen^{h_{HT}(D_S)}(D_S(-i)),
\]
we get that the natural map
\[
(\oplus_{a\leq i\leq b}\D^{h_{HT}(D_S)}_\Sen(D_S(-i)))^\Gamma\otimes_{K\otimes_{\Q}S}(K_{n}\otimes_{\Q}S)[t,t^{-1}]\longrightarrow \oplus_{i\in\mathbb{Z}}\D_\Sen^{n}(D_S(-i))
\]
is an isomorphism. Taking $\Gamma$-invariants on both sides, we get
\[
(\oplus_{a\leq i\leq b}\D^{h_{HT}(D_S)}_\Sen(D_S(-i)))^\Gamma=(\oplus_{a\leq i\leq b}\D^{n}_\Sen(D_S(-i)))^\Gamma.
\]
This yields the lemma.
\end{proof}

\begin{remark}
If $D_S$ is Hodge-Tate with weights in $[a,b]$, then by taking $\Gamma$-invariants on both sides of (\ref{eq:def-HT}), we see that $\D^n_\Sen(D_S(-i))^{\Gamma}=0$ for any $n\geq h_{HT}(D_S)$ and $i\notin [a,b]$.
\end{remark}

\begin{lemma}\label{lem:HT}
If $D_S$ is a Hodge-Tate $\m$-module over $\mathbf{B}^\dag_{\rig,K}\widehat{\otimes}_{\Q}S$ with weights in $[a,b]$, then for any morphism $S\ra R$ of affinoid algebras over $K$, $D_R$ is Hodge-Tate with weights in $[a,b]$ and $h_{HT}(D_R)\leq h_{HT}(D_S)$. Furthermore, the natural map 
\[
\D^n_\Sen(D_S(i))^\Gamma\otimes_{S}R\ra\D^n_\Sen(D_R(i))^\Gamma
\] 
is an isomorphism for any $i\in\mathbb{Z}$ and $n\geq h_{HT}(D_S)$. As a consequence, the natural map $D_{\mathrm{HT}}(D_S)\otimes_SR\ra D_{\mathrm{HT}}(D_R)$ is an isomorphism.
\end{lemma}
\begin{proof}
Let $n\geq h_{HT}(D_S)$. Tensoring with $R$ over $S$ on both sides of (\ref{eq:def-HT}), we get that the natural map
\[
(\oplus_{a\leq i\leq b}\D^n_\Sen(D_S(-i))^\Gamma\otimes_SR)\otimes_{K\otimes_{\Q}R}(K_n\otimes_{\Q}R)[t,t^{-1}]\longrightarrow \oplus_{i\in\mathbb{Z}}\D_\Sen^n(D_R(-i)).
\]
is an isomorphism. Comparing $\Gamma$-invariants on both sides, we get that the natural map
\[
\D^n_\Sen(D_S(-i))^\Gamma\otimes_{S}R\ra\D^n_\Sen(D_R(-i))^\Gamma
\]
is an isomorphism for any $a\leq i\leq b$. This implies that the natural map
\[
(\oplus_{a\leq i\leq b}\D^n_\Sen(D_R(-i))^\Gamma\otimes_{K\otimes_{\Q}R}(K_n\otimes_{\Q}R)[t,t^{-1}]\longrightarrow \oplus_{i\in\mathbb{Z}}\D_\Sen^n(D_R(-i)).
\]
is an isomorphism. This proves the lemma.
\end{proof}

\begin{cor}
If $D_S$ is a Hodge-Tate $\m$-module of rank $d$ over $\mathbf{B}_{\rig,K}^\dag\widehat{\otimes}_{\Q}S$, then $D_{\mathrm{HT}}(D_S)$ is a locally free coherent $K\otimes_{\Q}S$-module of rank $d$.
\end{cor}
\begin{proof}
By the previous lemma, it suffices to treat the case that $S$ is a finite extension of $K$; this is clear from the isomorphism (\ref{eq:def-HT}).
\end{proof}

\begin{defn}
Let $X$ be a rigid analytic space over $F$, and let $D_X$ be a family of $\m$-modules of rank $d$ over $X$. We call $D_X$ \emph{Hodge-Tate} with weights in $[a,b]$ if for some (hence any) admissible cover $\{M(S_i)\}_{i\in I}$ of $X$, $D_{S_i}$ is Hodge-Tate with weights in $[a,b]$ for any $i\in I$. We define $D_{\mathrm{HT}}(D_X)$ to be the gluing of all $D_{\mathrm{HT}}(D_{S_i})$'s.
\end{defn}

\begin{lemma}\label{lem:HT-criterion}
Let $D_S$ be a $\m$-module over $\mathbf{B}_{\rig,K}^\dag\widehat{\otimes}_{\Q}S$. Then (\ref{eq:def-HT}) is an isomorphism if and only if the natural map
\begin{equation}\label{eq:lem-HT}
\oplus_{a\leq i\leq b}\D_\Sen^n(D_S)^{\Gamma_n=\chi^i}\longrightarrow\D_\Sen^n(D_S)
\end{equation}
is an isomorphism. Furthermore, if this is the case, then (\ref{eq:def-HT}) holds for $n$.
\end{lemma}
\begin{proof}
For the ``$\Rightarrow$'' part, since (\ref{eq:def-HT}) is an isomorphism, we deduce that
\begin{equation}\label{eq:lem-HT-2}
\D_\Sen^n(D_S)=\oplus_{a\leq i\leq b}t^i\cdot\D^n_{\Sen}(D_S(-i))^\Gamma
\otimes_{K\otimes_{\Q}S}(K_n\otimes_{\Q}S).
\end{equation}
Note that $t^i\cdot\D^n_{\Sen}(D_S(-i))^\Gamma\subseteq\D_\Sen^n(D_S)^{\Gamma_n=\chi^i}$. Hence (\ref{eq:lem-HT-2}) implies that (\ref{eq:lem-HT}) is surjective. On the other hand, it is clear that (\ref{eq:def-HT}) is injective; hence it is an isomorphism. Conversely, suppose that (\ref{eq:lem-HT}) is an isomorphism. Note that
\[
\D_\Sen^n(D_S)^{\Gamma_n=\chi^i}=t^i\cdot\D_\Sen^n(D_S(-i))^{\Gamma_n}=(t^i\cdot\D_\Sen^n(D_S(-i))^\Gamma)
\otimes_{K\otimes_{\Q}S}(K_n\otimes_{\Q}S),
\]
where the latter equality follows from \cite[Proposition 2.2.1]{BC07}. This implies that $D_S$ satisfies (\ref{eq:lem-HT-2}), yielding that $D_S$ satisfies (\ref{eq:def-HT}).
\end{proof}

\begin{prop}\label{prop:HT-family}
Let $S$ be reduced, and let $D_S$ be a $\m$-module over $\mathbf{B}_{\rig,K}^\dag\widehat{\otimes}S$. Suppose that there exists a Zariski dense subset $Z\subset M(S)$ such that $D_z$ is Hodge-Tate with weights in $[a,b]$ for any $z\in Z$ and $\sup_{z\in Z}\{h_{HT}(D_z)\}<\infty$. Then $D_S$ is Hodge-Tate with weights in $[a,b]$.
\end{prop}
\begin{proof}
Let $n\geq\sup_{z\in Z}\{h_{HT}(D_z)\}$ such that $D_S^n$ is defined, and let $\gamma$ be a topological generator of $\Gamma_n$. For any $a\leq i\leq b$, let $p_i$ denote the operator
$\prod_{a\leq j\leq b, j\neq i}\frac{\gamma-\chi^{j}(\gamma)}{\chi^i(\gamma)-\chi^j(\gamma)}$,
and let $M_i=p_i(\D_\Sen^n(D_S))$. It is clear that $p_i$ is the identity on $\D_{\Sen}^n(D_S)^{\Gamma_n=\chi^i}$; hence  $\D_{\Sen}^n(D_S)^{\Gamma_n=\chi^i}\subseteq M_i$. On the other hand, for any $z\in Z$, since $D_z$ is Hodge-Tate with weights in $[a,b]$ and $h_{HT}(D_z)\leq n$, we deduce from Lemma \ref{lem:HT-criterion} that $p_i(\D_\Sen^n(D_z))=\D^n_\Sen(D_z)^{\Gamma_n=\chi^i}$. This implies that $M_i$ maps onto $\D^n_\Sen(D_z)^{\Gamma_n=\chi^i}$ under the specialization $\D_\Sen^n(D_S)\ra \D_\Sen^n(D_z)$. Since $S$ is reduced and $Z$ is Zariski dense, we conclude $M_i\subseteq\D^n_\Sen(D)^{\Gamma_n=\chi^i}$; hence $M_i=\D^n_\Sen(D)^{\Gamma_n=\chi^i}$.
Let $M=\oplus_{a\leq i\leq b}M_i$. We claim that the natural inclusion $M\subseteq \D_\Sen^n(D_S)$ is an isomorphism. In fact, for any $z\in Z$, since $\D_\Sen^n(D_z)=\oplus_{a\leq i\leq b}\D_\Sen^n(D_z)^{\Gamma_n=\chi^i}$, we have that $M$ maps onto $\D_\Sen^n(D_z)$. Thus $\D^n_\Sen(D_S)/M$ vanishes at $z$. We therefore conclude $\D^n_\Sen(D_S)/M=0$ because $S$ is reduced and $Z$ is Zariski dense. By Lemma \ref{lem:HT-criterion} and the claim, we conclude that $D_S$ is Hodge-Tate with weights in $[a,b]$.
\end{proof}

\section{Families of de Rham $\m$-modules}
\begin{defn}\label{def:dR}
Let $D_S$ be a $\m$-module of rank $d$ over $\mathbf{B}_{\rig,K}^\dag\widehat{\otimes}_{\Q}S$. For any positive integer $n$, if $D_S^{r_n}$ is defined, then we equip $\D_\dif^n(D_S)$ with the filtration $\mathrm{Fil}^i\D_\dif^n(D_S)=t^i\D_\dif^{+,n}(D_S)$. We call $D_S$ \emph{de Rham with weights in $[a,b]$} if there exists a positive integer $n$ such that
\begin{enumerate}
\item[(1)]
the natural map
\begin{equation}\label{eq:def-de Rham}
\D^n_\dif(D_S)^\Gamma\otimes_{K\otimes_{\Q}S}(K_n\otimes_{\Q}S)[[t]][1/t]\longrightarrow \D_\dif^n(D_S)
\end{equation}
is an isomorphism;
\item[(2)]$\mathrm{Fil}^{-b}(\D^n_\dif(D_S)^\Gamma)=D_S$ and $\mathrm{Fil}^{-a+1}(\D^n_\dif(D_S)^\Gamma)=0$
where $\mathrm{Fil}^{i}(\D^n_\dif(D_S)^\Gamma)$ is the induced filtration on $\D^n_\dif(D_S)^\Gamma$.
\end{enumerate}
We denote by $h_{dR}(D_S)$ the smallest $n$ which satisfies these conditions, and we define $D_{\mathrm{dR}}(D_S)=\D^{h_{dR}(D_S)}_\dif(D_S)^\Gamma$.
\end{defn}

\begin{lemma}\label{lem:dR-inv}
Let $D$ be a de Rham $\m$-module over $\mathbf{B}_{\rig,K}^\dag\widehat{\otimes}_{\Q}S$. Then for any $n\geq h_{dR}(D_S)$, $\D^n_\dif(D_S)^\Gamma=D_{\mathrm{dR}}(D_S)$
\end{lemma}
\begin{proof}
We tensor $K_{n+1}\otimes_{\Q}S[[t]][1/t]$ on both sides of the map
\[
\D^{h_{dR}(D_S)}_\dif(D_S)^\Gamma\otimes_{K\otimes_{\Q}S}(K_{h_{dR}(D_S)}\otimes_{\Q}S)[[t]][1/t]\longrightarrow \D_\dif^{h_{dR}(D_S)}(D_S),
\]
yielding that the map
\[
\D^{h_{dR}(D_S)}_\dif(D_S)^\Gamma\otimes_{K\otimes_{\Q}S}(K_{n}\otimes_{\Q}S)[[t]][1/t]\longrightarrow \D_\dif^{n}(D_S).
\]
is an isomorphism. Comparing $\Gamma$-invariants on both sides, we get the desired result.
\end{proof}

\begin{lemma}\label{lem:dR-HT}
If $D$ is a de Rham $\m$-module of rank $d$ over $\mathbf{B}_{\rig,K}^\dag\widehat{\otimes}_{\Q}S$ with weights in $[a,b]$, then $D$ is Hodge-Tate with weights in $[a,b]$ and $h_{HT}(D_S)\leq h_{dR}(D_S)$. Furthermore, we have $\mathrm{Gr}^iD_{\mathrm{dR}}(D_S)=\D_\Sen^n(D_S(i))^\Gamma$ under the identification $\mathrm{Gr}^i\D_\dif^n(D_S)=\D_\Sen^n(D_S(i))$ for any $n\geq h_{dR}(D_S)$.
\end{lemma}
\begin{proof}
Let $n\geq h_{dR}(D_S)$. Since (\ref{eq:def-de Rham}) is an isomorphism, we deduce that the natural map of graded modules
\begin{equation}\label{eq:lem-dR-HT}
\oplus_{i\in\mathbb{Z}}\mathrm{Gr}^iD_{\mathrm{dR}}(D_S)
\otimes_{K\otimes_{\Q}S}(K_n\otimes_{\Q}S)[t,t^{-1}]\longrightarrow \oplus_{i\in\mathbb{Z}}\D_\Sen^n(D_S(i))
\end{equation}
is surjective. On the other hand, since $t^i\cdot\mathrm{Gr}^{-i}D_{\mathrm{dR}}(D_S)\subset \D_{\Sen}^n(D_S)$, we have that the natural map
\[
\oplus_{a\leq i\leq b}t^i\cdot\mathrm{Gr}^{-i}D_{\mathrm{dR}}(D_S)\ra \D_\Sen^n(D_S)
\]
is injective. This implies that (\ref{eq:lem-dR-HT}) is injective; hence it is an isomorphism. Comparing the $\Gamma$-invariants on both sides, we get $\mathrm{Gr}^iD_{\mathrm{dR}}(D_S)=\D_\Sen^n(D_S(i))^\Gamma$ for each $i\in\mathbb{Z}$. This proves the lemma.
\end{proof}

\begin{lemma}\label{lem:dR}
If $D_S$ is a de Rham $\m$-module over $\mathbf{B}^\dag_{\rig,K}\widehat{\otimes}_{\Q}S$, then for any morphism $S\ra R$ of affinoid algebras over $K$, $D_R$ is de Rham with weights in $[a,b]$ and $h_{dR}(D_R)\leq h_{dR}(D_S)$. Furthermore, the natural maps $\mathrm{Fil}^i D_{\mathrm{dR}}(D_S)\otimes_{S}R\ra \mathrm{Fil}^iD_{\mathrm{dR}}(D_R)$ are isomorphisms for all $i\in \mathbb{Z}$.
\end{lemma}
\begin{proof}
Let $n\geq h_{dR}(D_S)$. Tensoring with $(K_n\otimes_{\Q}R)[[t]][1/t]$ on both sides of (\ref{eq:def-de Rham}), we get that the natural map
\begin{equation}\label{eq:lem-dR}
(\D^n_\dif(D_S)^\Gamma\otimes_S R)\otimes_{K\otimes_{\Q}R}(K_n\otimes_{\Q}R)[[t]][1/t]\longrightarrow \D_\dif^n(D_R).
\end{equation}
is an isomorphism. Comparing $\Gamma$-invariants on both sides of (\ref{eq:lem-dR}), we get that the natural map $\D^n_\dif(D_S)^\Gamma\otimes_{S}R\ra\D^n_\dif(D_R)^\Gamma$
is an isomorphism; hence $D_R$ is de Rham. Then by Lemmas \ref{lem:HT} and \ref{lem:dR-HT}, we deduce that the natural map
$\mathrm{Gr}^i(D_{\mathrm{dR}}(D_S))\otimes_SR\ra\mathrm{Gr}^i(D_{\mathrm{dR}}(D_R))$ is an isomorphism.
This implies the rest of the lemma.
\end{proof}

\begin{cor}
If $D_S$ is a de Rham $\m$-module of rank $d$ over $\mathbf{B}_{\rig,K}^\dag\widehat{\otimes}_{\Q} S$, then $D_{\mathrm{dR}}(D_S)$ is a locally free coherent $K\otimes_{\Q}S$-module of rank $d$.
\end{cor}
\begin{proof}
We first note that for each $i\in\mathbb{Z}$, $\mathrm{Gr}^i(D_{\mathrm{dR}}(D_S))$, which is isomorphic to $\D_\Sen^n(D_S(i))^\Gamma$ by Lemma \ref{lem:dR-HT}, is a coherent $K\otimes_{\Q}S$-module. We then deduce that $D_{\mathrm{dR}}(D_S)$ is a coherent $K\otimes_{\Q}S$-module. Using Lemma \ref{lem:dR}, it then suffices to treat the case that $S$ is a finite extension of $K$; this follows easily from the isomorphism (\ref{eq:def-de Rham}).
\end{proof}

\begin{defn}
Let $X$ be a rigid analytic space over $F$, and let $D_X$ be a family of $\m$-modules of rank $d$ over $X$. We call $D_X$ \emph{de Rham} if for some (hence any) admissible cover $\{M(S_i)\}_{i\in I}$ of $X$, $D_{S_i}$ is de Rham with weights in $[a,b]$ for any $i\in I$. We define $D_{\mathrm{dR}}(D_X)$ to be the gluing of all $D_{\mathrm{dR}}(D_{S_i})$'s.
\end{defn}

\begin{lemma}\label{lem:dR-weight}
If $D_S$ is a de Rham $\m$-module over $\mathbf{B}^\dag_{\rig,K}\widehat{\otimes}_{\Q}S$ of rank $d$ with weights in $[a,b]$, then $t^{-a}\D_\dif^{+,n}(D_S)\subset D_{\mathrm{dR}}(D_S)\otimes_{K\otimes_{\Q}S}(K_n\otimes_{\Q}S)[[t]]\subset t^{-b}\D_\dif^{+,n}(D_S)$ for any $n\geq h_{dR}(D_S)$.
\end{lemma}
\begin{proof}
Since $\mathrm{Gr}^{-b}D_{\mathrm{dR}}(D_S)=D_{\mathrm{dR}}(D_S)$, we get $D_{\mathrm{dR}}(D_S)\subset t^{-b}\D^{+,n}_\dif(D_S)$; hence $D_{\mathrm{dR}}(D_S)\otimes_{K\otimes_{\Q}S}(K_n\otimes_{\Q}S)[[t]]\subset t^{-b}\D_\dif^{+,n}(D_S)$. By the proof of Lemma \ref{lem:dR-HT}, we know that the natural map (\ref{eq:lem-dR-HT}) is an isomorphism of graded modules. By the facts that $\mathrm{Gr}^iD_{\mathrm{dR}}(D_S)=0$ for $i\geq -a+1$ and $\mathrm{Fil}^i\D_\dif^n(D_S)$ is $t$-adically complete, we thus deduce that $t^{-a}\D_\dif^{+,n}(D_S)\subset D_{\mathrm{dR}}(D_S)\otimes_{K\otimes_{\Q}S}(K_n\otimes_{\Q}S)[[t]]$.
\end{proof}

\begin{lemma}
Let $D_S$ be a Hodge-Tate $\m$-module over $\mathbf{B}^\dag_{\rig,K}\widehat{\otimes}_{\Q}S$ with weights in $[a,b]$. Then for any $k\geq b-a+1$, $i\in[a,b]$, $n\geq h_{HT}(D_S)$ and $\gamma\in\Gamma_n$, the map $\gamma-\chi^i(\gamma):t^k\D_\dif^{+,n}(D_S)\ra t^k\D_\dif^{+,n}(D_S)$ is bijective.
\end{lemma}
\begin{proof}
Since $\D_\dif^{+,n}(D_S)$ is $t$-adically complete, it suffices to show that
\[
\gamma-1:t^k\D_\dif^{+,n}(D_S)/t^{k+1}\D_\dif^{+,n}(D_S)\ra t^k\D_\dif^{+,n}(D_S)/t^{k+1}\D_\dif^{+,n}(D_S)
\]
is bijective for any $k\geq b-a+1$. Note that $t^k\D_\dif^{+,n}(D_S)/t^{k+1}\D_\dif^{+,n}(D_S)$ is isomorphic to $\D_\Sen^n(D_S(k))$ as a $\Gamma$-module. Note that $\D^n_\Sen(D_S(k))=\oplus_{a\leq j\leq b}(\D^n_\Sen(D_S))^{\Gamma_n=\chi^{j+k}}$ by Lemma \ref{lem:HT-criterion}. Since $j+k\geq b+1$ for all $j\in [a,b]$, we deduce that $\gamma-\chi^i(\gamma)$ is bijective on $\D^n_\Sen(D_S(k))$.
\end{proof}

\begin{lemma}\label{lem:dR-criterion}
Let $D_S$ be a Hodge-Tate $\m$-module over $\mathbf{B}^\dag_{\rig,K}\widehat{\otimes}_{\Q}S$ with weights in $[a,b]$. Then $D_S$ is de Rham if and only if there exists a positive integer $n\geq h_{HT}(D_S)$ such that $\Pi_{i=a}^{2b-a}(\gamma-\chi(\gamma)^i)\D_\dif^{+,n}(D_S)\subset t^{b-a+1}\D_\dif^{+,n}(D_S)$. Furthermore, if this is the case, then (\ref{eq:def-de Rham}) holds for $n$.
\end{lemma}
\begin{proof}
Suppose that $D_S$ is de Rham. Let $n\geq h_{dR}(D_S)$, and put
\[
N=D_{\mathrm{dR}}(D_S)\otimes_{K\otimes_{\Q}S}(K_n\otimes_{\Q}S)[[t]].
\]
Since $D$ has weights in $[a,b]$, by Lemma \ref{lem:dR-weight}, we have $t^{-a}\D_\dif^{+,n}(D_S)\subset N\subset t^{-b}\D_\dif^{+,n}(D_S)$. On the other hand, by the construction of $N$, it is clear that $(\gamma-1)N\subset tN$. It therefore follows that
\[
\Pi_{i=a}^{2b-a}(\gamma-\chi(\gamma)^i)\D_\dif^{+,n}(D_S)\subset
\Pi_{i=a}^{2b-a}(\gamma-\chi(\gamma)^i)(t^aN)\subset t^{2b-a+1}N\subset t^{b-a+1}\D_\dif^{+,n}(D_S).
\]
Now suppose $\Pi_{i=a}^{2b-a}(\gamma-\chi(\gamma)^i)\D_\dif^{+,n}(D_S)\subset t^{b-a+1}\D_\dif^{+,n}(D_S)$ for some $n\geq h_{HT}(D_S)$. We claim that for any $j\in[a,b]$ and $a\in(\D^n_\Sen(D_S))^{\Gamma_n=\chi^j}$, we can lift $a$ to an element in $(\D_{\dif}^{+,n}(D_S))^{\Gamma_n=\chi^j}$. In fact, let $\tilde{a}$ be any lift of $a$ in $\D_\dif^{+,n}(D_S)$, and let 
\[
\tilde{b}=\prod_{a\leq i\leq 2b-a, i\neq j}\frac{\gamma-\chi^i(\gamma)}{\chi^j(\gamma)-\chi^i(\gamma)}\tilde{a}
\]
where $\gamma$ is a topological generator of $\Gamma_n$; it is clear that $\tilde{b}$ is also a lift of $a$.  Furthermore, by assumption, we have $(\gamma-\chi^j(\gamma))(\tilde{b})\in \Pi_{i=a}^{2b-a}(\gamma-\chi(\gamma)^i)\D_\dif^{+,n}(D_S)\subset t^{b-a+1}\D^{+,n}_\dif(D_S)$. By the previous lemma, we choose some $\tilde{c}\in t^{b-a+1}\D^{+,n}_\dif(D_S)$ satisfying $(\gamma-\chi^j(\gamma))(\tilde{b})=(\gamma-\chi^j(\gamma))(\tilde{c})$. It is then clear that $\tilde{b}-\tilde{c}$ is a desired lift of $a$. Since $\D^n_\Sen(D_S)=\oplus_{a\leq i\leq b}(\D^n_\Sen(D_S))^{\Gamma_n=\chi^i}$, we have that $(\D^n_\Sen(D_S))^{\Gamma_n=\chi^i}$ is locally free for each $i\in[a,b]$. By shrinking $M(S)$, we may further suppose that each $(\D^n_\Sen(D_S))^{\Gamma_n=\chi^i}$ is free. We then deduce from the claim that there exists a free $K_n\otimes_{\Q}S$-module $M\subseteq(\D_\dif^{n}(D_S))^{\Gamma_n}$ such that the natural map
\[
M\otimes_{K_n\otimes_{\Q}S}(K_n\otimes_{\Q}S)[[t]][1/t]\longrightarrow \D_\dif^n(D_S)
\]
is an isomorphism. It follows that the natural map
\[
M^\Gamma\otimes_{K\otimes_{\Q}S}(K_n\otimes_{\Q}S)[[t]][1/t]\longrightarrow \D_\dif^n(D_S)
\]
is an isomorphism because $M=M^{\Gamma}\otimes_{K\otimes_{\Q}S}(K_n\otimes_{\Q}S)$ by \cite[Proposition 2.2.1]{BC07}. Taking $\Gamma$-invariants on both sides, we get $M^{\Gamma}=(\D_\dif^n(D_S))^\Gamma$. This implies that $D_S$ is de Rham.
\end{proof}

\begin{prop}\label{prop:dR-family}
Let $S$ be reduced, and Let $D_S$ be a $\m$-module over $\mathbf{B}_{\rig,K}^\dag\widehat{\otimes}S$. Suppose that there exists a Zariski dense subset $Z\subset M(S)$ such that $D_z$ is de Rham with weights in $[a,b]$ for any $z\in Z$ and $\sup_{z\in Z}\{h_{dR}(D_z)\}<\infty$. Then $D_S$ is de Rham with weights in $[a,b]$.
\end{prop}
\begin{proof}
By Proposition \ref{prop:HT-family}, we first have that $D_S$ is Hodge-Tate with weights in $[a,b]$. Let $n\geq \max\{h_{HT}(D_S),\sup_{z\in Z}\{h_{dR}(D_z)\}\}$. By Lemma \ref{lem:dR-criterion}, we have
\[
\Pi_{i=a}^{2b-a}(\gamma-\chi(\gamma)^i)\D_\dif^{+,n}(D_z)\subset t^{b-a+1}\D_\dif^{+,n}(D_z)
\]
for any $z\in Z$. This implies $\Pi_{i=a}^{2b-a}(\gamma-\chi(\gamma)^i)\D_\dif^{+,n}(D_S)\subset t^{b-a+1}\D_\dif^{+,n}(D_S)$ because $S$ is reduced and $Z$ is Zariski dense. Hence $D_S$ is de Rham by Lemma \ref{lem:dR-criterion} again.
\end{proof}
\begin{remark}
The work presented in this paper was finished in the summer of 2012 and made public at the beginning of 2013. After that, the work of Bellovin \cite{Bel13} appeared in the summer of 2013 in which she builded up a more robust theory of families of Hodge-Tate and de Rham representations over rigid analytic spaces.  First of all, she generalized Berger's dictionary, which relates Fontaine's functors to $\m$-modules, to families of $p$-adic representations \cite[Theorem 1.1.1]{Bel13}. This result implies that our theory of families of Hodge-Tate and de Rham $\m$-modules with bounded Hodge-Tate weights developed in \S3 and \S4 can be viewed as a generalization of Berger and Colmez's theory of families of Hodge-Tate and de Rham representations with bounded Hodge-Tate weights. Moreover, she developed a theory of families of ``partial" Hodge-Tate and de Rham representations with bounded Hodge-Tate weights. That is, the periods of the fibers are assumed to be of some constant rank which is not necessarily equal to the rank of the family. In addition, she removes the ``reduced" assumption on the base by considering all artinian points.  We refer the reader to \cite{Bel13} for more results and details.
\end{remark}

\section{$P$-adic local monodromy for families of de Rham $\m$-modules}
The main goal of this section is to prove the $p$-adic local monodromy for the restrictions of families of de Rham $\m$-modules with bounded Hodge-Tate weights on their Shilov boundary. The proof is modeled on Berger-Colmez's proof of the $p$-adic local monodromy for families of de Rham representations with bounded Hodge-Tate weights \cite[\S6]{BC07}.  Recall that $\nabla=\log(\gamma)/\log(\chi(\gamma))$ which is independent of the choice of $\gamma\in\Gamma$. This gives rise to an action of the Lie algebra of $\Gamma$ on  $\m$-modules and their localizations. In the following, we fix $E$ to be a finite extension of the products of the complete residue fields of the Shilov boundary of $M(S)$.

\begin{prop}\label{prop:N_dR}
Let $D_S$ be a de Rham $\m$-module of rank $d$ over $\mathbf{B}_{\rig,K}^\dag\widehat{\otimes}_{\Q}S$ with weights in $[a,b]$. For any $s>0$ such that $n(s)\geq h_{dR}(D_S)$, let
\[
N_s(D_E)=\{y\in t^{-b}D^{s}_E\hspace{2mm}\text{such that}\hspace{2mm}\iota_n(y)\in D_{\mathrm{dR}}(D_S)\otimes_{K\otimes_{\Q}S}(K_n\otimes_{\Q}E)[[t]]\hspace{1mm}\text{for each}\hspace{2mm}n\geq n(s)\}.
\]
Then the following are true.
\begin{enumerate}
\item[(1)]The $\mathbf{B}^{\dag,s}_{\rig,K}\widehat{\otimes}_{\Q} E$-module $N_s(D_E)$ is free of rank $d$ and stable under $\Gamma$.
\item[(2)]We have
$N_s(D_E)\otimes_{\mathbf{B}^{\dag,s}_{\rig,K}\widehat{\otimes}_{\Q}E,\iota_n}(K_n\otimes_{\Q}E)[[t]]
=D_{\mathrm{dR}}(D_S)\otimes_{K\otimes_{\Q}S}(K_n\otimes_{\Q}E)[[t]]$ for each $n\geq n(s)$.
\end{enumerate}
Furthermore, if we put $N_{\mathrm{dR}}(D_E)=N_s(D_E)\otimes_{\mathbf{B}^{\dag,s}_{\rig,K}\widehat{\otimes}_{\Q}E}
\mathbf{B}^{\dag}_{\rig,K}\widehat{\otimes}_{\Q}E$, then the following are true.
\begin{enumerate}
\item[(3)]The $\mathbf{B}^{\dag}_{\rig,K}\widehat{\otimes}_{\Q} E$-module $N_{\mathrm{dR}}(D_E)$ is free of rank $d$, stable under $\Gamma$, and independent of the choice of $s$.
\item[(4)]We have $\varphi^*(N_{\mathrm{dR}}(D_E))=N_{\mathrm{dR}}(D_E)$ and $\nabla(N_{\mathrm{dR}}(D_E))\subset t\cdot N_{\mathrm{dR}}(D_E)$.
\end{enumerate}
\end{prop}
\begin{proof}
First note that the sequence of $K_n\otimes_{\Q}E[[t]]$-modules
\[
\{D_{\mathrm{dR}}(D_S)\otimes_{K\otimes_{\Q}S}(K_n\otimes_{\Q}E)[[t]]\}_{n\geq n(s)}
\]
is $\varphi$-compatible in the sense of \cite[D\'efinition II.1.1]{LB04}. Then by the proof of \cite[Th\'eor\`eme II.1.2]{LB04} (using the fact that $E$ is a finite product of $p$-adic local fields which are endowed with discrete valuations extending the standard one on $\Q$), we see that $N_{\mathrm{dR}}(D_E)$ is the unique $\m$-module $M_E$ contained in $D_E[1/t]$ such that 
\[
M_E^s\otimes_{\mathbf{B}^{\dag,s}_{\rig,K}\widehat{\otimes}_{\Q}E,\iota_n}(K_n\otimes_{\Q}E)[[t]]
=D_{\mathrm{dR}}(D_S)\otimes_{K\otimes_{\Q}S}(K_n\otimes_{\Q}E)[[t]]
\]  
for any $n\geq n(s)$. Furthermore, the proof of [\emph{loc. cit.}] implies the proposition except the second half of (4). To see that, 
note that
\[
\iota_n(\nabla(N_s(D_E)))=\nabla(\iota_n(N_s(D_E)))\subset tD_{\mathrm{dR}}(D_S)\otimes_{K\otimes_{\Q}S}(K_n\otimes_{\Q}E)[[t]].
\] 
This yields that $\nabla(N_s(D_E))\subset t N_s(D_E)$ for all $s$. Hence $\nabla(N_{\mathrm{dR}}(D_E))\subset tN_{\mathrm{dR}}(D_E)$ as $N_{\mathrm{dR}}(D_E)$ is equal to the union of all $N_s(D_E)$. 
\end{proof}


\begin{prop}\label{prop:monodromy}
Keep notations as in Proposition \ref{prop:N_dR}. Then there exists a finite extension $L$ over $K$ such that
\[
M=(N_{\mathrm{dR}}(D_E)\otimes_{\mathbf{B}_{\rig,K}^\dag\widehat{\otimes}_{\Q}E}
\mathbf{B}_{\log,L}^\dag\widehat{\otimes}_{\Q}E)^{I_L}
\]
is a free $L_0'\otimes_{\Q}E$-module of rank $d$ and the natural map
\begin{equation*}
\begin{split}
M\otimes_{L_0'\otimes_{\Q}E}
\mathbf{B}_{\log,L}^\dag\widehat{\otimes}_{\Q}E
\longrightarrow N_{\mathrm{dR}}(D_E)\otimes_{\mathbf{B}_{\rig,K}^\dag\widehat{\otimes}_{\Q}E}
\mathbf{B}_{\log,L}^\dag\widehat{\otimes}_{\Q}E
\end{split}
\end{equation*}
is an isomorphism.
\end{prop}
\begin{proof}
Let $f'=[K_0':\Q]$. Note that there is a canonical decomposition 
\[
\mathbf{B}_{\rig,K}^\dag\widehat{\otimes}_{\Q}E\cong\prod_{i=0}^{f'-1}\r_E^{(i)}
\]
where each $\r_E^{(i)}$ is isomorphic to $\r_E$ and stable under $\Gamma_K$, and satisfies $\varphi(\r_E^{(i)})\subset\r_E^{(i+1)}$ ($\r_E^{(f')}=\r_E^{(0)}$). Let $N^{(i)}_{\mathrm{dR}}(D_E)=N_{\mathrm{dR}}(D_E)
\otimes_{\mathbf{B}_{\rig,K}^\dag\widehat{\otimes}_{\Q}E}\r_E^{(i)}$. It follows that each $N_{\mathrm{dR}}^{(i)}(D_E)$ is stable under $\partial=\nabla/t$ and $\varphi^{f'}$; hence it is a $p$-adic differential equation with a Frobenius structure. By the versions of the $p$-adic local monodromy theorem proved by Andr\'e \cite{An} or Mebkhout \cite{Meb}, we conclude that each $N^{(i)}_{\mathrm{dR}}(D_E)$ is potentially unipotent. This yields the proposition using the argument of \cite[Proposition 6.2.2]{BC07} and \cite[Corollaire 6.2.3]{BC07}.
\end{proof}

\begin{lemma}\label{lem:monodromy}
Keep notations as in Proposition \ref{prop:monodromy}, and let
\[
M=(N_s(D_E)\otimes_{\mathbf{B}_{\rig,K}^{\dag,s}\widehat{\otimes}_{\Q}E}
\mathbf{B}_{\log,K}^{\dag,s}\widehat{\otimes}_{\Q}E)^{I_L}
\]
for sufficiently large $s$. Then for any $n\geq n(s)$, we have
\begin{equation}\label{eq:lem-monodromy}
L\otimes_{L_0}\iota_n(M)=(\D_\dif(D_E\otimes_{\mathbf{B}_{\rig,K}^\dag\widehat{\otimes}_{\Q}E}
\mathbf{B}_{\rig,L}^\dag\widehat{\otimes}_{\Q}E))^{I_L}
\end{equation}
\end{lemma}
\begin{proof}
By the previous proposition, the left hand side of (\ref{eq:lem-monodromy}) is a free $L\otimes_{L_0}L_0'\otimes_{\Q}E$-module of rank $d$. On the other hand, since $((L_n\otimes_{\Q}E)[[t]][1/t])^{I_L}=L\otimes_{L_0}L_0'\otimes_{\Q}E$, we deduce that the right hand side of (\ref{eq:lem-monodromy}), which obviously contains the left hand side, is an $L\otimes_{L_0}L_0'\otimes_{\Q}E$-module generated by at most $d$-elements. Using the fact that $L\otimes_{L_0}L_0'\otimes_{\Q}E$ is a product of fields, we deduce the desired identity.
\end{proof}

\section{Proof of the main theorem}
Now let $V_X$ be a finite slope family of dimension $d$ over a reduced rigid analytic space $X$ over $F$. We start by making some preliminary reductions. After a finite surjective base change of $X$, we may assume that $Q(T)$ factors as $\prod_{i=1}^m(T-F_i)$. By reordering the $F_i$'s and throwing away some points of $Z$ we may further assume that for all $z\in Z$, $v_p(F_i(z))\geq v_p(F_j(z))$ if $i>j$ and $F_i(z)\neq F_j(z)$ if $F_i\neq F_j$. We then set $\F_{i,z}=
D_{\mathrm{st}}^+(V_z)^{(\varphi^f-F_1(z))\cdots(\varphi^f-F_{i}(z))=0}$ for all $z\in Z$ and $1\leq i\leq m$.
Using Definition \ref{def:fs}(3), we may suppose that $\F_{i,z}\subseteq \F_z$ for all $z\in Z$ and $1\leq i\leq m$ by shrinking $Z$. Furthermore, by the fact that $N\varphi=p\varphi N$ and the condition that $v_p(F_i(z))\geq v_p(F_j(z))$ if $i>j$, we see that $N=0$ on each graded piece $\F_{i,z}/\F_{i-1,z}$.
Let $c_{i,z}$ be the rank of $ \F_{i,z}/\F_{i-1,z}$ over $K_0\otimes k(z)$, and partition $Z$ into finitely many subsets according to the sequence $c_{i,z}$. Note that at least one of these subsets of $Z$ has to be Zariski dense. Replace $Z$ by this subset and set $c_i = c_{i,z}$ for $z\in Z$.

For $z\in Z$, we will inductively set $\m$-submodules $\mathrm{Fil}_{i,z}\subset\D_\rig^\dag(V_z)$ for $1\leq i\leq m$ such that $D^+_{\mathrm{st}}(\mathrm{Fil}_{i,z})=\F_{i,z}$. For $i=1$, since $V_z$ has non-positive Hodge-Tate weights and $N(\F_{1,z})=0$, we have
\[
\F_{1,z}=D^+_{\mathrm{st}}(V_z)^{\varphi^f=F_1(z), N=0}=D^+_{\mathrm{crys}}(V_z)^{\varphi^f=F_1(z)}=\D_\rig^\dag(V_z)^{\Gamma=1,\varphi^f=F_z(z)}
\]
using Berger's dictionary \cite[Th\'eor\`eme 3.6]{LB02}. Let $\mathrm{Fil}_{1,z}$ be the saturation of the $\m$-submodule of $\D_\rig^\dag(V_z)$ generated by $\mathcal{F}_{1,z}$. It is then clear that $D^+_{\mathrm{st}}(\mathrm{Fil}_{1,z})=D_{\mathrm{crys}}^+(\mathrm{Fil}_{1,z})=\F_{1,z}$.  Now suppose we have set $\mathrm{Fil}_{i-1,z}$ for some $i\geq 2$ such that $D^+_{\mathrm{st}}(\mathrm{Fil}_{i-1,z})=\F_{i-1,z}$. It follows that
\[
D_{\mathrm{st}}^+(\D_\rig^\dag(V_z)/\mathrm{Fil}_{i-1,z})=D_{\mathrm{st}}^+(V_z)/\F_{i-1,z}.
\]
Note that
\[
\F_{i,z}/\F_{i-1,z}=(D_{\mathrm{st}}^+(V_z)/\F_{i-1,z})^{\varphi^f=F_{i}(z),N=0}.
\]
Hence
\[
\F_{i,z}/\F_{i-1,z}=D^+_{\mathrm{crys}}(\D_\rig^\dag(V_z)/\mathrm{Fil}_{i,z})^{\varphi^f=F_{i}(z)}\subset
(\D_\rig^\dag(V_z)/\mathrm{Fil}_{i-1,z})^\Gamma.
\]
We set $\mathrm{Fil}_{i,z}$ to be the preimage of the saturation of the $\m$-submodule of $\D_\rig^\dag(V_z)/\mathrm{Fil}_{i-1,z}$ generated by $\F_{i,z}/\F_{i-1,z}$.
Now for each $1\leq i\leq m$, we define the character $\delta_i:K^\times\ra\OO(X)^\times$ by setting $\delta_i(p)=F_i^{-1}$ and $\delta_i(\OO_K^\times)=1$. Let $D_X=\D_\rig^\dag(V_X)^{\vee}$.

\begin{lemma}\label{lem:de Rham-part}
Suppose that $X$ is irreducible. Then for each $0\leq i\leq m$, there exists a proper birational morphism $\pi:X'\ra X$ and a sub-family of $\m$-modules $D^{(i)}_{X'}\subset D_{X'}$ over $X'$ of rank $d-c_1-\dots-c_i$ such that
\begin{enumerate}
\item[(1)]
for any $x\in X'$, the natural map $D_x^{(i)}\ra D_x$ is injective;
\item[(2)]
there exists a Zariski open dense subset $U$ of $X'$ such that for any $z\in Z'=\pi^{-1}(Z)\cap U$, the natural map $D^{(i)}_z\ra D_z$ is the dual of the projection $\D_\rig^\dag(V_{\pi(z)})\ra \D_\rig^\dag(V_{\pi(z)})/\mathrm{Fil}_{i,\pi(z)}$.
\end{enumerate}
\end{lemma}
\begin{proof}
We proceed by induction on $i$. The initial case is trivial. Suppose that for some $1\leq i\leq m$, the lemma is true for $i-1$.
Note that $\mathcal{F}_{i,z}/\mathcal{F}_{i-1,z}$ maps into $\D_\rig^\dag(V_{z})/\mathrm{Fil}_{i,z}$ for any $z\in Z$. Since $\F_{i,z}/\F_{i-1,z}=(D_{\mathrm{crys}}^+(V_z)/\F_{i-1,z})^{\varphi^f=F_{i}(z)}$, we get that $(D^{(i)}_z)^{\vee}(\pi^{*}(\delta_i)(z))$ has $k(z)$-dimension $c_i$ for any $z\in Z'$. Since $Z'$ is Zariski dense in $X'$, by Proposition \ref{prop:cohomology}, after adapting $X'$ and $U$, we may find a sub-family of $\m$-modules $D^{(i)}_{X'}$ of $D^{(i-1)}_{X'}$ with rank $d-c_1-\dots-c_i$ such that
\begin{enumerate}
\item[(1')]$D_x^{(i)}\ra D_x^{(i-1)}$ is injective for any $x\in X'$;
\item[(2')]for any $z\in \pi^{-1}(Z)\cap U$, $D_z^{(i)}$ is the kernel of the dual of the map
\[
    (\mathbf{B}_{\rig,K}^\dag\otimes_{\Q}k(z))\cdot(\mathcal{F}_{i,\pi(z)}/\mathcal{F}_{i-1,\pi(z)})\ra \D_\rig^\dag(V_{\pi(z)})/\mathrm{Fil}_{i,\pi(z)}.
\]
\end{enumerate}
It is clear that (1') and (2') imply (1) and (2) respectively; this finishes the inductive step.
\end{proof}

To prove Theorem \ref{thm:main}, we also need the following lemma.
\begin{lemma}\label{lem:tate-sen}
Let $V_S$ be a free $S$-linear representation of $G_K$ of rank $d$. Then there exists a positive integer $m(V_S)$ such that for any $x\in M(S)$ and $a\in\D_\dif^{+}(V_x)$, if $a$ is $\Gamma$-invariant, then $a\in\D_\dif^{+,m(V_S)}(V_x)$.
\end{lemma}
\begin{proof}
This is a consequence of the Tate-Sen method. Using \cite[Th\'eor\`{e}me 4.2.9]{BC07}, we first choose a finite extension $L$ over $K$ and some positive integer $m$ so that $\D_{\rig,L}^{\dag,r_m}(V_S)$ is a free $\mathbf{B}_{\rig,L}^{\dag,r_m}\widehat{\otimes}_{\Q}S$-module with a basis $\mathrm{e}=(e_1,\dots,e_d)$. Let $\gamma$ be a topological generator of $\Gamma_{L_m}$ and write $\gamma(e)=eG$ for some $G\in\mathrm{GL}_d(\mathbf{B}_{\rig,L}^{\dag,r_m}\widehat{\otimes}_{\Q}S)$. Recall that by the classical work of Tate \cite{T}, we know that there exists a constant $c>0$ such that $v_p((\gamma-1)x)\leq v_p(x)+c$ for any nonzero $x\in (1-R_{L,m})\widehat{L}_\infty$, where $R_{L,m}:\widehat{L}_\infty\ra L_m$ is Tate's normalized trace map. Since the localization map $\iota_m:\mathbf{B}_{\rig,L}^{\dag,r_m}\ra L_m[[t]]$ is continuous, by enlarging $m$, we may suppose that the constant term of $\iota_m(G)-1$ has norm less than $p^{-c}$. We fix some $m_0\in\mathbb{N}$ such that $K_\infty\cap L_m=K_{m_0}\cap L_m$.

Now let $a\in\D_\dif^{+,K_n}(V_x)^\Gamma$ for some $x\in X$ and $n\geq m$. We will show that $a\in\D_\dif^{+,K_{m_0}}(V_x)^\Gamma$. Since $\iota_m(\mathrm{e})$ forms a basis of $\D^{+,L_n}_{\dif}(V_S)$, we may write $a=\iota_m(\mathrm{e})(x)A$ for some
\[
A\in \mathrm{M}_{d\times1}((L_n\otimes_{\Q}k(x))[[t]]).
\]
The $\Gamma$-invariance of $a$ implies $\iota_m(G(x))\gamma(A)=A$; thus 
\[
(1-R_{L,m})\iota_m(G(x))\gamma(A)=(1-R_{L,m})A.
\]  
Note that $\iota_m(G(x))$ has entries in $(L_m\otimes_{\Q}k(x))[[t]]$.  It follows that $(G(x)-1)B=(1-\gamma^{-1})B$ where $B=(1-R_{L,m})A$. Let $B_0$ be the constant term of $B$. If $B_0\neq0$, then the constant term of $(\iota_m(G(x))-1)B$ has valuation $\geq v(\iota_m(G(x))-1)+v(B_0)>v(B_0)+c$  whereas the constant term $(1-\gamma^{-1})B_0$ of $(1-\gamma^{-1})B$ has valuation $\leq v(B_0)+c$; this yields a contradiction. Hence $B_0=0$. Iterating this argument, we get $B=0$. Hence $a\in \D_\dif^{+,L_m}(V_x)\cap\D_\dif^{+,K_n}(V_x)\subset\D_\dif^{+,K_{m_0}}(V_x)$. Thus we may choose $m(V_S)=m_0$.
\end{proof}

\begin{remark}
Although we do not need this in this paper, it is worthwhile to point out that the argument of Lemma \ref{lem:tate-sen} works equally for families of $\m$-modules and even a sequence of $\varphi$-compatible $K_n[[t]]\widehat{\otimes}_{\Q}S$-modules $\{M_n\}_{n}$ in the vein of \cite[D\'efinition II.1.1]{LB04}. That is, each $M_n$ is a finite projective $K_n[[t]]\widehat{\otimes}_{\Q}S$-module equipped with a continuous $K_n[[t]]$-semilinear and $S$-linear $\Gamma$-action, and satisfies $M_n\otimes_{K_n[[t]]\widehat{\otimes}_{\Q}S} K_{n+1}[[t]]\widehat{\otimes}_{\Q}S\cong M_{n+1}$.
\end{remark}

\subsection*{Proof of Theorem 0.3}
We retain the notations as above. By passing to irreducible components, we may suppose that $X$ is irreducible. We then apply Lemma \ref{lem:de Rham-part} to $V_X$. Note that $V_{X'}$ is again a finite slope family over $X'$ with the Zariski dense set  of semi-stable points $\pi^{-1}(Z)$. We may suppose that $X'=X$. Let $\lambda:\D^\dag_{\rig}(V_X)=D^{\vee}_X\ra (D_X^{(m)})^{\vee}$ be the dual of $D_X^{(m)}\ra D_X$, and let $P_X=\ker(\lambda)$. For any $x\in X$, since $D^{(m)}_x\ra D_x$ is injective, we get that the image of $\lambda_x$ is a $\m$-submodule of rank $d-c_1-\cdots-c_m$. Thus by Lemma \ref{lem:ker-birational}, after adapting $X$, we may assume that $P_X$ is a family of $\m$-modules of rank $c_1+\cdots+c_m$, and there exists a Zariski open dense subset $U\subset X$ such that $P_x=\ker(\lambda_x)$ for any $x\in U$. Note that $\ker(\lambda_z)=\mathrm{Fil}_{i,z}$ for any $z\in Z$. Thus by replacing $Z$ with $Z\cap U$, we may assume that $P_z=\mathrm{Fil}_{i,z}$ for any $z\in Z$. We claim that $P_{X}$ is de Rham with weights in $[-b,0]$. To do so, we set $Y$ to be the set of $x\in X$ for which $P_x$ is de Rham with weights in $[a,b]$. By the previous lemma, we see that for any affinoid subdomain $M(S)\subset X$, there exists an integer $m(V_S)$ such that if $P_x$ is de Rham for some $x\in M(S)$, then $h_{dR}(P_x)\leq m(V_S)$. We then deduce from Proposition \ref{prop:dR-family} that $Y\cap M(S)$ is a Zariski closed subset of $M(S)$. Hence $Y$ is a Zariski closed subset of $X$. On the other hand, since $P_z$ is de Rham with weights in $[-b,0]$, we get $Z\subset Y$; thus $Y=X$ by the Zariski density of $Z$. Furthermore, using Proposition \ref{prop:dR-family} and the previous lemma again, we deduce that $P_X$ is de Rham with weights in $[-b,0]$. As a consequence, we obtain a locally free coherent $\OO_X\otimes_{\Q}K$-module $D_{\mathrm{dR}}(P_X)$ of rank $c_1+\cdots+c_m$.

The next step is to show that for any $x\in X$, $D_{\mathrm{dR}}(P_x)$ is contained in $D^+_{\mathrm{st}}(V_x)\otimes_{K_0}K$. Let $Y$ be the set of $x\in X$ satisfying this condition. We first show that $Y$ is a Zariski closed subset of $X$. For this, it suffices to show that $Y\cap M(S)$ is a Zariski closed subset of $M(S)$ for any affinoid subdomain $M(S)$ of $X$. To show this, we employ the $p$-adic local monodromy for families of de Rham $\m$-modules. As in \S5, let $E$ be the product of the complete residue fields of the Shilov boundary of $M(S)$. Since $P_S$ is a family of de Rham $\m$-modules with weights in $[-b,0]$, by Lemma \ref{lem:monodromy}, there exists a finite extension $L$ of $K$ such that for sufficiently large $s$ and $n\geq n(s)$, we have
\[
L\otimes_{L_0}\iota_n(M)=(\D_\dif(P_E\otimes_{\mathbf{B}_{\rig,K}^\dag\widehat{\otimes}_{\Q}E}
\mathbf{B}_{\rig,L}^\dag\widehat{\otimes}_{\Q}E))^{I_L}
\]
for
$M=(N_s(P_E)\otimes_{\mathbf{B}_{\rig,K}^{\dag,s}\widehat{\otimes}_{\Q}E}
\mathbf{B}_{\log,K}^{\dag,s}\widehat{\otimes}_{\Q}E)^{I_L}$; furthermore, $N_s(P_E)\subset P_E^{s}$. Thus
\[
\iota_n(M)\subset \iota_n(P_E\otimes_{\mathbf{B}_{\rig,K}^{\dag,s}\widehat{\otimes}_{\Q}E}
\mathbf{B}_{\log,K}^{\dag,s}\widehat{\otimes}_{\Q}E)\subset
\iota_n(\D_\rig^\dag(V_E)\otimes_{\mathbf{B}_{\rig,K}^{\dag,s}\widehat{\otimes}_{\Q}E}
\mathbf{B}_{\log,K}^{\dag,s}\widehat{\otimes}_{\Q}E)\subset\mathbf{B}^+_{\mathrm{st}}\widehat{\otimes}_{\Q}V_E.
\]
Note that $D_{\mathrm{dR}}(P_E)\subset \D_\dif^+(P_E)\subset\D_\dif^+(V_E)\subset\mathbf{B}_{\mathrm{dR}}^+\widehat{\otimes}_{\Q}V_E$. This yields
\[
D_{\mathrm{dR}}(P_E)\subset (\mathbf{B}^+_{\mathrm{st}}\widehat{\otimes}_{\Q}V_E)\otimes_{L_0}L\cap \mathbf{B}_{\mathrm{dR}}^+\widehat{\otimes}_{\Q}V_E=
(\mathbf{B}^+_{\mathrm{st}}\widehat{\otimes}_{\Q}V_E)\otimes_{L_0}L.
\]
We therefore deduce from \cite[Lemme 6.3.1]{BC07} that
\[
D_{\mathrm{dR}}(P_S)\subset (\mathbf{B}^+_{\mathrm{st}}\widehat{\otimes}_{\Q}V_E)\otimes_{L_0}L\cap
\mathbf{B}_{\mathrm{dR}}^+\widehat{\otimes}_{\Q}V_S=(\mathbf{B}^+_{\mathrm{st}}\widehat{\otimes}_{\Q}V_S)\otimes_{L_0}L.
\]
It follows that $Y\cap M(S)$, which is the set of $x\in M(S)$ such that $D_{\mathrm{dR}}(P_x)\subset (\mathbf{B}^+_{\mathrm{st}}\otimes_{\Q}V_x)\otimes_{K_0}K$, is Zariski closed in $M(S)$.

To conclude the theorem, it then suffices to show that $D_{\mathrm{dR}}(P_x)\subset (D^+_{\mathrm{st}}(V_x)\otimes_{K_0}K)^{Q(\varphi^f)(x)=0}$ for any $x\in X$; here we $K$-linearly extend the $\varphi^f$-action to $D^+_{\mathrm{st}}(V_x)\otimes_{K_0}K$. Note that $\mathrm{Fil}_{m,z}$ is semi-stable with $D_{\mathrm{st}}(\mathrm{Fil}_{m,z})=\mathcal{F}_{m,z}$. This implies that $Q(\varphi^f)(D_{\mathrm{dR}}(P_X))$ vanishes at $z$, yielding that $Q(\varphi^f)(D_{\mathrm{dR}}(P_X))=0$ by the Zariski density of $Z$.

\end{document}